\documentclass[11pt,twoside]{amsart} 
\title{On z-superstable and critical configurations of chip-firing pairs}
 \author{Zach Benton}
 \address{Stanford University}
 \email{zbenton@stanford.edu}
 \author{Jane Kwak}
 \address{University of California, Los Angeles}
 \email{janekwak1@g.ucla.edu}
 \author{Suho Oh}
 \address{Texas State University}
 \email{suhooh@txstate.edu}
 \author{Mateo Torres}
 \address{University of Delaware}
 \email{mtorres@udel.edu}
 \author{Mckinley Xie}
 \address{Texas A\&M University}
 \email{mckinleyxie@tamu.edu}
\usepackage{prefix} % for final paper just copy and paste the sty here
\date{}
\begin{document}
    \begin{abstract}
    It is well known that there is a duality map between superstable configurations and critical configurations of a graph. This was extended to all M-matrices in (Guzm\`an-Klivans 2015). We show a natural way to extend this to all $(L,M)$-chip firing pairs introduced in (Guzm\`an-Klivans 2016). In addition, we study various properties of this map.        
    \end{abstract}
    \maketitle
    \section{Introduction}

Chip-firing is a game that takes place on a connected graph $G$, where chips are distributed across the vertices of $G$ and moved to adjacent vertices based on a straightforward rule. This dynamical system has a profound theory that links to various fields in mathematics and physics \cite{BakTangWiesenfeld88}, \cite{Dhar90}, \cite{Gabrielov93}, \cite{digraph}. For further details, please see the recent textbooks \cite{Klivans} and \cite{CorryPerk}.

Consider a simple graph $G$ with $n + 1$ vertices, where one vertex is designated as the \emph{sink}, and chips are placed on each non-sink vertex. The allocation of chips is described by an integer vector $\vec{c} \in \mathbb{Z}^{n}$, known as a \textit{(chip) configuration}. A non-sink vertex with chips equal to or greater than its degree can fire, distributing one chip along each incident edge to adjacent vertices. A configuration $\vec{c}$ is termed \emph{stable} if no non-sink vertex is able to fire.

For a connected graph $G$, any starting configuration will eventually reach stability via sequence of firings, with some chips moving to the sink vertex. Letting $v_1, ..., v_{n}$ represent the non-sink vertices of $G$, the chip-firing rules can be described using $L_G$, the $n \times n$ \textit{reduced Laplacian} matrix of $G$.

The outcome of firing a vertex $v_i$ on a chip configuration $\vec{c}$ is represented by $\vec{c} - L\vec{e_i}$, where $\vec{e_i}$ denotes the $i$-th standard basis vector. The matrix $L_G$ establishes an equivalence relation among the vectors in ${\mathbb Z}^{n}$, with $\vec{c}$ and $\vec{d}$ being \emph{firing equivalent} if $\vec{c} - \vec{d}$ lies within the image of $L_G$. This determines the \emph{critical group} of $G$, as described by $\mathcal{K}(G) := {\mathbb Z}^n/\Img L_G$ \cite{Biggs99}.

A configuration $\vec{c}$ is considered \emph{valid} (or \emph{effective}) if for all $i = 1, \dots, n$, the condition $c_i \geq 0$ holds. The goal is to identify notable valid configurations within each equivalence class $[\vec{c}] \in {\mathcal K}(G)$. It turns out that each $[\vec{c}]$ contains a distinct valid configuration that is \emph{critical}, which means it is stable and can be derived by iteratively firing from a sufficiently large configuration $\vec{b}$. Stabilizing the sum of two critical configurations results in a critical configuration.

One can show that each $[\vec{c}]$ contains a unique valid configuration that is \emph{superstable}, which means that it is stable under \emph{set-firings}, the simultaneous firing of a set of vertices.  A superstable configuration represents a solution to an energy minimization problem and aligns with the concept of a \emph{$G$-parking function}. 

For a connected graph $G$, there is a straightforward bijection linking the set of critical configurations and the set of superstable configurations, which are both correspondingly in bijection with the set of spanning trees of $G$. This simple map (taking a configuration, negating it coordinate-wise from a certain maximal configuration) is what is called the \newword{duality map} between superstable configurations and critical configurations, and our focus is to extend this map to more general models.

Recently, chip-firing has been extended to more general settings, where the reduced Laplacian of a graph is replaced by other matrices (see, for instance, \cite{gabrielov} and \cite{GuzKliMmatrices}). We use a matrix $M$ to define a firing rule that mimics the graphical setting: firing $v_i$ now takes a configuration $\vec{c}$ to $\vec{c} - M \vec{e_i}$, where $\vec{e_i}$ stands for the unit vector with $1$ at the $i$-th coordinate. For a well-defined notion of chip-firing we require that $M$ satisfies an \emph{avalanche finite} property, so that repeated firings of any initial configuration eventually stabilize in an appropriate sense.  The class of matrices with this property are known as \emph{$M$-matrices}, and can be characterized in a number of ways (see \cref{def:Mmatrix} below). In \cite{GuzKliMmatrices}, Guzm\'an and Klivans have shown that the chip-firing theory defined by an $M$-matrix leads to good notions of critical and superstable configurations. They further generalized this model by introducing an invertible matrix $L$, calling $(L,M)$ a \emph{chip-firing pair}, and extending the definition of critical and superstable configurations to that model \cite{GuzKlivans}.

%%%TODO : add GK (L,M)-pair paper as reference

Our goal is to extend the duality between critical and superstable configurations to $(L,M)$-chip firing pairs. This will answer Question~5.2 of \cite{cho2024} in a much more general sense, since signed graphs are special cases of the chip-firing pair model. In Section~$2$ we summarize the Guzm\'an-Klivans theory of chip-firing pairs and also review the previously known duality between superstable and critical configurations of $M$-matrices. In Section~$3 $ we provide our main result on extending the duality map to chip-firing pairs. In Section~$4$ we introduce the notion of frackets and use them to study some properties of the map discussed in the main result. %then study some properties of this map.

%An important subtlety here is that one must consider \emph{multiset} firings and the notion of `$z$-superstability', we review these concepts below. 

\section{Prerequisites}
In this section we review the definition of chip-firing pairs. After that we review the duality map between superstable configurations and critical configurations for $M$-matrices. Then we will go over the tools developed in \cite{cho2024} that we will use to deal with the $z$-superstable and critical configurations of chip-firing pairs.

\subsection{Chip firing pairs}
In \cite{GuzKliMmatrices}, Guzm\'an and Klivans generalized the chip-firing on graphs to \newword{$M$-matrices}. $M$-matrices are used in various fields such as economics or scientific computing \cite{BurmanErn05}, \cite{CiarletRaviart73}, \cite{Leontief41}, \cite{Plemmons77}. Guzm\'an and Klivans further generalized this by introducing an invertible integer matrix $L$ in \newword{chip-firing pairs} $(L,M)$ introduced in \cite{GuzKlivans}. 

\begin{definition}
\label{def:Mmatrix}
Suppose $M$ is an $n \times n$ matrix such that $(M)_{ii} > 0$ for all $i$ and $(M)_{ij} \leq 0$ for all $i \neq j$. Then $M$ is called an (invertible) \emph{$M$-matrix} if any of the following equivalent conditions hold:
\begin{enumerate}
    \item $M$ is avalanche finite;
    \item The real part of the eigenvalues of $M$ are positive;
    \item The entries of $M^{-1}$ are non-negative;
    \item There exists a vector $\vec{x} \in {\mathbb R}^{n}$ with ${\vec x} \geq \vec{0}$ such that $M \vec{x}$ has all positive entries.
\end{enumerate}
\end{definition}

The pair $(L,M)$, an $M$-matrix $M$ together with an invertible integer matrix $L$, is called a \emph{chip-firing pair}. The relevant (chip) \emph{configurations} $\vec{c} \in \mathbb{Z}^{n}$ are simply integer vectors with $n$ entries, and chip-firing is dictated by the matrix $L$. In particular, $(M,M)$ recovers the chip-firing on $M$-matrices and $(L_G,L_G)$ when $L_G$ is the (reduced) Laplacian of a graph recovers the classical chip-firing model on graphs.

\begin{remark}
In this paper, we only focus on integral $M$-matrices (since the duality map for $M$-matrices given in \cite{GuzKliMmatrices} that we extend upon, is only given for integral matrices).
\end{remark}

\begin{definition}\label{defn:valid}
Suppose $(L,M)$ is a chip-firing pair.  A configuration $\vec{c}$ is \emph{valid} if $\vec{c} \in S^+$, where
\[S^+ = \{LM^{-1} \vec{x} : LM^{-1}\vec{x} \in \mathbb{Z}^{n}, \vec{x} \in \mathbb{R}^n_{\geq 0}\}. \]

Equivalently, a configuration $\vec{c}$ is valid if $ML^{-1}\vec{c} \in R^+$, where
\[R^+ = \{ \vec{x} \in \mathbb{R}^n_{\geq 0} : LM^{-1}\vec{x} \in \mathbb{Z}^{n} \}.\]
\end{definition}

In particular, for $(M,M)$, being valid is exactly the same as being a nonnegative integer vector. 

\begin{definition}
Suppose $(L,M)$ is a chip-firing pair, and suppose that $\vec{c} \in S^+$ is a valid configuration. A site $i \in \{1,\dots, n\}$ is \emph{ready to fire} if
\[\vec{c} - L\vec{e}_i \in S^+,\]
so that the vector obtained by subtracting the $i$th row of $L$ from $\vec{c}$ is also valid.

Similarly, suppose $\vec{x} \in R^+$. Then a site $i \in \{1,\dots, n\}$ is \emph{ready to fire} if 
\[\vec{x} - M\vec{e}_i \in R^+.\]

A configuration $\vec{c}$ (in $S^+$ or $R^+$) is \emph{stable} if no site is ready to fire.
\end{definition}

If $i$ is ready to fire, we declare that $\vec{b} = \vec{c} - L \vec{e}_i \in S^+$ is derived from $\vec{c}$ through a \emph{legal firing}. Repeating this process, a vector $\vec{a} \in S^+$ is said to be derived from $\vec{c}$ through a \emph{sequence of legal firings}. For a configuration $\vec{c} \in S^+$ (or conversely, $\vec{d} \in R^+$), we define $\text{stab}_{S^+}(\vec{c})$ (and $\text{stab}_{R^+}(\vec{d})$) as the resulting configuration after executing a series of legal firings until no site remains eligible to fire. Adapting the proof presented in \cite[Theorem 2.2.2]{Klivans}, it can be established that both $\text{stab}_{S^+}(\vec{c})$ and $\text{stab}_{R^+}(\vec{d})$ are uniquely determined. When it is clear whether we are dealing with $S^{+}$ or $R^{+}$, we use $\text{stab}(\vec{x})$ to refer to the stabilization of the configuration $\vec{x}$.

%In this context the notion of a \emph{valid} configuration takes on a new meaning.

The definition of critical and superstable configurations in this model are as follows.

\begin{definition}
        Given an $(L,M)$ pair, a configuration $\vec c \in S^+$ is reachable if there exists some configuration $\vec d \in S^+$ satisfying:

        \begin{itemize}
            \item $\vec d - L \vec e_i \in S^+$ for all $1 \leq i \leq n$
            \item $\vec c = \vec d - \sum_{j = 1}^k L \vec e_j$ and $\vec d - \sum_{j = 1}^\ell L \vec e_j \in S^+$ for all $\ell < k$.
        \end{itemize}
        Given an $(L,M)$ pair, a configuration $\vec c \in S^+$ is critical if $\vec c$ is both stable and reachable.
    \end{definition}

Critical configurations are those that are both stable and reachable by chip-firing from a sufficiently large configuration. They are useful because they index the equivalence classes of the critical group.

\begin{definition}[{\cite[Definition 4.3]{GuzKliMmatrices}}]
A vector $f \in \ZZ^n$ with $f \geq 0$ is \newword{$z$-superstable} if for every $z \in \ZZ^n$ with $z \geq 0$ and $z \neq 0$ there exists $1 \leq i \leq n$ such that $f_i - (Lz)_i < 0$.
\end{definition}

It turns out that in the equivalence class given by the matrix $L$ in $S^{+}$, we can always find a unique representative that is critical and a unique representative that is $z$-superstable.

\begin{theorem}[{\cite[Theorems 3.5, 4.3, 5.5]{GuzKlivans}\label{thm:class}}]
Suppose $(L,M)$ is a chip-firing pair.  Then there exists exactly one $z$-superstable configuration and one critical configuration in each equivalence class $[\vec{c}]_L$.
\end{theorem}

\begin{remark}
For chip-firing pairs, there is the notion of a $\chi$-superstable configuration as well as a $z$-superstable configuration. From \cref{thm:class}, it is the $z$-superstable configurations that have the same size as the critical configurations. For the remainder of the paper, we will focus only on $z$-superstable configurations, and we will call them the superstable configurations of the chip-firing pair, omitting the letter $z$.
\end{remark}

In the next subsection, we go over the duality that is known to exist when $L=M$.

\subsection{Duality for \texorpdfstring{$M$}{M}-matrices}

If we take a chip-firing pair $(M,M)$, it recovers the chip-firing on $M$-matrices studied in \cite{GuzKliMmatrices}. Chip-firing on $M$-matrices generalizes many properties and results of the classical chip-firing on graphs, and one of them is the duality between superstable and critical configurations.

Given an $M$-matrix, it turns out that there is a critical configuration that is a coordinate-wise greater or equal to every other critical configuration. We call this configuration $\cm$, given by taking all diagonal entries of $M$ minus one and forming a vector (in the classical case, this corresponds to having $\deg(v)-1$ chips for each vertex $v$). 

\begin{theorem}[\cite{GuzKliMmatrices}]
\label{thm:usualMduality}
Let $M$ be an $M$-matrix. Let $\cm$ denote the vector where each entry is coming from the corresponding diagonal entry $M_{ii}$ minus one. Then we have a bijection between superstable and critical configurations by the map $\vec{c} \rightarrow \cm - \vec{c}$.
\end{theorem}

    % go over crit/ superstable definitions again
    % go over critical group again
 
   \begin{example}
    \label{ex:unsignedkyle}
    %TODO : Write down L_G
    Consider the following graph that has the reduced Laplacian to be $L_G = \begin{pmatrix}
        3 & -1 & -1 \\
        -1 & 2 & -1 \\
        -1 & -1 & 3
    \end{pmatrix}.$
    %we have $\cmax = (2, 1, 2)$.
    
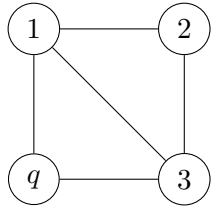
\begin{center}
        \begin{tikzpicture}[scale=2]
            \begin{scope}[every node/.style={circle,draw}]
\node (1) at (0,1) {$1$};
\node (2) at (1,1) {$2$};
\node (3) at (1,0) {$3$};
\node (q) at (0,0) {$q$};
\end{scope}
% positive edges
    \begin{scope}[every node/.style={draw=black},
every edge/.style={draw=black}]
\path (1) edge (3);
\path (1) edge (q);
\path (3) edge (2);
\path (3) edge (q);
\path (1) edge (2);
\end{scope}
% negative edges

        \end{tikzpicture}
    \end{center}
    The superstable configurations and critical configurations are given in the following table:
    %In the table below, we can see that $\flip$ gives us a bijection between the superstables and criticals.
    \begin{center}\label{table:unsigned_toy_ss}
    \begin{tabular}{ c | c } 
        Superstables & Criticals \\
        \hline
        $(0, 0, 0)$ & $(2, 1, 2)$\\ 
        $(0, 0, 1)$ & $(2, 1, 1)$\\  
        $(0, 0, 2)$ & $(2, 1, 0)$\\
        $(0, 1, 0)$ & $(2, 0, 2)$\\
    \end{tabular}
    \begin{tabular}{ c | c }
        Superstables & Criticals \\
        \hline
        $(0, 1, 1)$ & $(2, 0, 1)$\\
        $(1, 0, 0)$ & $(1, 1, 2)$\\
        $(1, 1, 0)$ & $(1, 0, 2)$\\
        $(2, 0, 0)$ & $(0, 1, 2)$ 
    \end{tabular}
    \end{center}
    \end{example}

Notice that in the above example, we have a bijection between superstable configurations and critical configurations via the map $\vec{c} \rightarrow (2,1,2) - \vec{c}$. 

\begin{remark}
\label{rem:twobijects}
We have two different types of bijections that occur naturally between superstable configurations and critical configurations. The first is the one mentioned in \cref{thm:usualMduality} above. The second is the one obtained by mapping a superstable configuration to a unique critical configuration in the same equivalence class (guaranteed by \cref{thm:class}). In general, these maps are different, since the map $\vec{c} \rightarrow \cm - \vec{c}$ does not preserve the equivalence class in which one is located.
\end{remark}

Recall that our goal is to extend this to $(L,M)$ chip-firing pairs. The next subsection will show that the map $\vec{c} \rightarrow \cmax - \vec{c}$ does not work for chip-firing pairs.

    \subsection{The usual duality map does not work for chip-firing pairs}
    \label{subsec:classicdualnotwork}

Recall that the usual duality map between the superstable and 
critical configurations for graphs (and also $M$-matrices) is given by the map $\vec{c} \rightarrow \cm - \vec{c}$ for some fixed $\cm$. As can be seen in the example below, this does not work for $(L,M)$-pairs in general. The examples from this point throughout will be using $(L,M)$-pairs coming from a signed graph. The systematic study of signed graphs and their Laplacian was initiated by Zaslavsky in \cite{Zaslavsky} and also studied in \cite{MR666857}, \cite{MR4641710}, \cite{FundBapat}.

    \begin{example} \label{ex:signed_kyle} \label{table:signed_toy_ss}
    We take $L$ to be the (reduced) Laplacian of the following signed graph and $M$ to be the (reduced) Laplacian of the underlying unsigned graph. The Laplacian of the signed graph is simply obtained from the Laplacian of the underlying graph, by changing the signs of entries corresponding to negative edges. 

    %%PUT the matrix L here
    
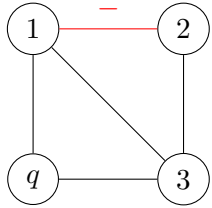
\begin{figure}[ht]
        \centering
        \begin{minipage}{0.3\linewidth}
            \begin{center}
            \begin{tikzpicture}[scale=2]
            \input{tikz/kyle-signed.tex}
            \end{tikzpicture}
            \end{center}
        \end{minipage}
        \begin{minipage}{0.6\linewidth}
            $$M = \begin{pmatrix}
        3 & -1 & -1 \\
        -1 & 2 & -1 \\
        -1 & -1 & 3
    \end{pmatrix} \hspace{5mm}
    L = \begin{pmatrix}
        3 & 1 & -1 \\
        1 & 2 & -1 \\
        -1 & -1 & 3
    \end{pmatrix}$$
        \end{minipage}
        \label{fig:Signed-Kyle-S+}
    \end{figure}

    %%Instead of preimages put configs of $S^+$
    
    \begin{center}
    Configurations in $S^+$
    
    \renewcommand{\arraystretch}{1.2}
    \begin{tabular}{ c | c } \label{table:toy_ss}
        Superstables & Criticals \\
        \hline
        $(0, 0, 0) $ & $(6, 4, 2) $ \\
        $(1, 1, 0) $ & $(7, 5, 2) $ \\	
        $(4, 3, 2) $ & $(8, 6, 0) $ \\
        $(5, 4, 2) $ & $(9, 7, 0) $ \\
    \end{tabular}
    \begin{tabular}{c | c}
    Superstables & Criticals \\\hline
        $(2, 2, 0) $ & $(8, 6, 2) $ \\	
        $(3, 3, 0) $ & $(9, 7, 2) $ \\	
        $(3, 2, 0) $ & $(6, 4, 1) $ \\
        $(4, 3, 0) $ & $(7, 5, 1) $ \\
    \end{tabular}
    \begin{tabular}{c | c}
    Superstables & Criticals \\\hline
        $(5, 4, 0) $ & $(8, 6, 1) $ \\
        $(6, 5, 0) $ & $(9, 7, 1) $ \\
        $(6, 4, 0) $ & $(6, 5, 2) $ \\
        $(7, 5, 0) $ & $(7, 6, 2)$
    \end{tabular}
    \end{center}

    \end{example}

Notice that in the table of superstable and critical configurations of the chip-firing pair, the coordinate-wise maximal critical configuration is $(9,7,2)$. However if we take the superstable configuration $(5,4,2)$, the vector we get by applying the traditional duality map $(9,7,2) - (5,4,2) = (4,3,0)$ is not a critical configuration.

Even worse, there are many cases where $c_{max}$, the critical configuration that has coordinate-wise maximal entries does not even exist.

\begin{example}
\label{ex:nocmax}
Consider the $(L,M)$-pair coming from the signed graph below. The underlying graph is $C_6$, the cycle on six vertices.

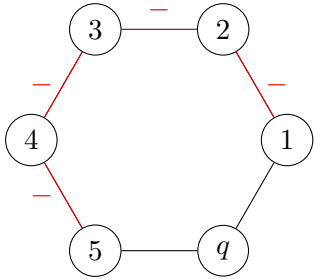
\begin{figure}[ht]
    \centering
    \renewcommand{\arraystretch}{1.2}
    \begin{minipage}{.45\linewidth}
        \centering
        \begin{tikzpicture}[scale=.85]
            \input{tikz/c6-signed}
        \end{tikzpicture}    
    \end{minipage}
    \begin{minipage}{.45\linewidth}
        \centering
        \begin{tabular}{ c } %\label{table:toy_ss}
        Critical configurations \\
        \hline
        (9, 15, 17, 15, 9) \\
        (12, 20, 23, 21, 13) \\
        (13, 21, 23, 20, 12) \\
        (7, 11, 12, 11, 7) \\
        (10, 16, 18, 17, 11) \\
        (11, 17, 18, 16, 10) \\
        \end{tabular}
    \end{minipage}
    \label{fig:Cycle-5}
\end{figure}

As can be checked from the table of critical configurations above, there is no critical configuration that is the maximal in all coordinates. 
\end{example}

\subsection{Finding the superstable/critical configurations of chip firing pairs.}
\label{subsec:lmpairstools}

In this subsection, we discuss an alternative way to find the superstable and critical configurations of $(L,M)$ chip-firing pairs, developed in \cite{cho2024}. %Moreover, from this point, whenever we say superstable configuration for $(L,M)$-pairs, we mean $z$-superstable configurations.

Let $\sstab(M)$ denote the set of superstable configurations of a $M$-matrix $M$ and let $\crit(M)$ denote the set of critical configurations. However, beware that we are not going to be using $\sstab(L,M)$ to denote the set of superstable configurations of $(L,M)$ and the same for $\crit(L,M)$. It turns out that for configurations in $S^{+}$, it is important to look at their preimages in $R^{+}$. Given any vector $\vec{f}$, we use $\floor{\vec{f}}$ to denote the vector obtained from $f$ by taking the floor at every coordinate.

    \begin{theorem}[{\cite[Theorem 3.2]{cho2024}}]
    \label{thm:22floor}
        Given an $(L, M)$ pair, a configuration $\vec c \in S^+$ is superstable/critical if and only if $\floor{ML^{-1} \vec c}$ is a  superstable/critical configuration of $M$.
    \end{theorem}

For example, we look at our running example coming from a signed graph.
%take a look at the the chip-firing pair coming from a signed graph in \cref{table:signed_toy_ss}. 

    \begin{example}
    \label{ex:usualbad}
    Consider the signed graph studied in  \cref{table:signed_toy_ss}. The table lists all superstable and critical configurations in $S^{+}$, their preimage in $R^{+}$, and the floor of the preimage. We can notice that the floor of the preimages are the superstable and critical configurations of the underlying graph we saw in \cref{ex:unsignedkyle} (however, not all superstable/critical configurations of the underlying graph are used).
    
    %the preimages of all the superstable configurations and critical configurations, and their floors. %It is easy to check that a vector $\vec{c}$ such that $\vec{v} %\rightarrow 
    %\vec{c} - \vec{v}$ gives a bijection between the superstables and criticals does not exists here: since we have $(0,0,0)$ as a superstable and $(8/3,11/6,0)$ is a maximal configuration (coordinate-wise) we must have $\vec{c} = (8/3,11/6,0)$, but then $(2/3,1/3,2)$ does not even map to a vector with nonnegative entries.

    %We can see that each superstable in $R^+$ floors to an unsigned superstable for this example. The same is true for the criticals. 
    
    
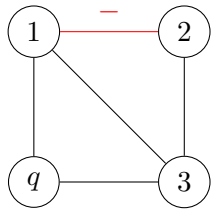
\begin{center}
        \begin{tikzpicture}[scale=2]
            % k4\e
\begin{scope}[every node/.style={circle,draw}]
\node (1) at (0,1) {$1$};
\node (2) at (1,1) {$2$};
\node (3) at (1,0) {$3$};
\node (q) at (0,0) {$q$};
\end{scope}
% positive edges
\begin{scope}[every node/.style={draw=black},
every edge/.style={draw=black}]
\path (1) edge (3);
\path (1) edge (q);
\path (3) edge (2);
\path (3) edge (q);
\end{scope}
% negative edges
\begin{scope}[every node/.style={text=red,font=\bfseries},
every edge/.style={draw=red}]
\path (1) edge node[above] {$-$} (2);
\end{scope}

        \end{tikzpicture}
    \end{center}

    %%% Superstable in S^{+}, preimage in R^{+}, floor, crit in S^{+}, preimage in R^{+}, floor.
    \begin{center}\label{table:toy_stats}
    \renewcommand{\arraystretch}{1.2}
    \begin{tabular}{ c | c | c | c | c | c}
$\LM(\sstab)$ & $\sstab$                         & $\floor{\sstab}$ & $\LM(\crit)$& $\crit$                             & $\floor{\crit}$\\\hline
$ (0, 0, 0) $ & $(0, 0, 0) $                     & $ (0, 0, 0) $    & $(6,4,2)   $& $ (2, 0, 2) $                       & $ (2, 0, 2) $ \\
$ (1, 1, 0) $ & $(0, \nicefrac{1}{2}, 0) $           & $ (0, 0, 0) $    & $(7,5,2)   $& $ (2, \nicefrac{1}{2}, 2) $             & $ (2, 0, 2) $ \\
$ (4, 3, 2) $ & $(\nicefrac{2}{3}, \nicefrac{1}{3}, 2) $ & $ (0, 0, 2) $    & $(8,6,0)   $& $ (\nicefrac{8}{3}, \nicefrac{4}{3}, 0) $   & $ (2, 1, 0) $ \\
$ (5, 4, 2) $ & $(\nicefrac{2}{3}, \nicefrac{5}{6}, 2) $ & $ (0, 0, 2) $    & $(9,7,0)   $& $ (\nicefrac{8}{3}, \nicefrac{11}{6}, 0)  $ & $ (2, 1, 0) $ \\
$ (2, 2, 0) $ & $(0, 1, 0) $                     & $ (0, 1, 0) $    & $(8,6,2)   $& $ (2, 1, 2) $                       & $ (2, 1, 2) $ \\
$ (3, 3, 0) $ & $(0, \nicefrac{3}{2}, 0) $           & $ (0, 1, 0) $    & $(9,7,2)   $& $ (2, \nicefrac{3}{2}, 2)  $            & $ (2, 1, 2) $ \\
$ (3, 2, 0) $ & $(\nicefrac{4}{3}, \nicefrac{1}{6}, 0) $ & $ (1, 0, 0) $    & $(6,4,1)   $& $ (\nicefrac{7}{3}, \nicefrac{1}{6}, 1)  $  & $ (2, 0, 1) $ \\
$ (4, 3, 0) $ & $(\nicefrac{4}{3}, \nicefrac{2}{3}, 0) $ & $ (1, 0, 0) $    & $(7,5,1)   $& $ (\nicefrac{7}{3}, \nicefrac{2}{3}, 1) $   & $ (2, 0, 1) $ \\
$ (5, 4, 0) $ & $(\nicefrac{4}{3}, \nicefrac{7}{6}, 0) $ & $ (1, 1, 0) $    & $(8,6,1)   $& $ (\nicefrac{7}{3}, \nicefrac{7}{6}, 1) $   & $ (2, 1, 1) $ \\
$ (6, 5, 0) $ & $(\nicefrac{4}{3}, \nicefrac{5}{3}, 0) $ & $ (1, 1, 0) $    & $(9,7,1)   $& $ (\nicefrac{7}{3}, \nicefrac{5}{3}, 1) $   & $ (2, 1, 1) $ \\
$ (6, 4, 0) $ & $(\nicefrac{8}{3}, \nicefrac{1}{3}, 0) $ & $ (2, 0, 0) $    & $(6,5,2)   $& $ (\nicefrac{2}{3}, \nicefrac{4}{3}, 2) $   & $ (0, 1, 2) $ \\
$ (7, 5, 0) $ & $(\nicefrac{8}{3}, \nicefrac{5}{6}, 0) $ & $ (2, 0, 0) $    & $(7,6,2)   $& $ (\nicefrac{2}{3}, \nicefrac{11}{6}, 2) $  & $ (0, 1, 2) $
    \end{tabular}
    \end{center}
    \end{example}

Thanks to \cref{thm:22floor}, it is much more convenient to deal with the preimages of configurations, especially when trying to check if it is superstable or critical. Given a superstable configuration in $S^{+}$, we will denote its preimage in $R^{+}$ as \newword{superstable preimage} and for a critical configuration in $S^{+}$, we are going to denote its preimage in $R^{+}$ as \newword{critical preimage}. We are also going to use $\sstab(L,M)$ to denote the set of superstable preimages of a $(L,M)$ chip-firing pair and use $\crit(L,M)$ to denote the set of critical preimages.

    \section{The Duality map for chip-firing pairs}

In this section we establish the duality between the superstable configurations and critical configurations of $(L,M)$-pairs, that extends the canonical duality between superstable and critical configurations of $M$-matrices. We are mainly going to be dealing with the preimages of the configurations in $R^{+}$. The idea is to group up the preimages that have the same floor and then map such groupings into different groupings.

%In order to find a map between superstable buckets and critical buckets, we define an involution map on the set of superstable configurations of $M$. 

\begin{remark}
    Recall that for chip-firing pairs, the coordinate-wise maximal critical configuration does not necessarily exist, as was the case in \cref{ex:nocmax}. However, for $M$-matrices it does. Throughout the paper, given any $(L,M)$-pair and we write $\cm$, it stands for $\cm$ of $M$.
\end{remark}

Consider the map $\sst(M) \rightarrow \sst(M)$ that sends $\vec s \in \sst(M)$ to $\sstab(\cmax-\vec s)$, where $\sstab(\vec{v})$ means that we are taking the unique superstable configuration in the equivalence class that contains $\vec{v}$ (we use $\crit(\vec{v})$ for the unique critical configuration in the same class). This map relies only on $M$ and is completely independent of $L$. This map sends a superstable configuration to the unique superstable configuration in the same equivalence class (under $M$) as the image of $\vec s$ under the usual duality map. Now we are going to use the information of $L$, to modify this map. For any vector $\vec s$, we will use $\{\vec s\}$ to denote its fractional part: $\{\vec s\} := \vec s - \floor{\vec s}.$

\begin{definition}
\label{def:ivomap}
    For any chip-firing pair $(L,M)$, we define the map $\ivo : \sst(M) \rightarrow \sst(M)$:
    \[
    \ivo(\vec s) = \begin{cases}
     \vec s &\text{if $\{\LM2\vec s\:\} = \{\LM\cm\}$, } \\
    \sstab(\cmax-\vec s) &\text{otherwise.}
    \end{cases}
    \]
    
\end{definition}

%It is clear from definition that the above map is indeed an involution:

\begin{proposition}
\label{prop:ivoinvol}
    The map $\ivo$ is an involution.
\end{proposition}
\begin{proof}
    If $\{\LM2\vec s\:\} = \{\LM\cm\}$, we have the identity map. If that condition does not hold, then we have
    \begin{align*}
    \ivo(\ivo(\vec s)) &= \sstab(\cm-\sstab(\cm-\vec s))\\
    % this step is a little bit loosey goosey maybe prove this earlier or elaborate
    &= \sstab(\cm-(\cm-\vec s))\\
    &= \sstab(\vec s)\\
    &= \vec s
    \end{align*}
    Therefore, $\ivo(\ivo(\vec s)) = \vec s$.
\end{proof}

In the special case where $L=M$ (this recovers the usual chip-firing on $M$-matrices, and in particular when $M$ is the Laplacian of a graph, the classical chip-firing) the above involution is simply the identity map: since $\LM2\vec s$ is an integer vector for any integer vector $\vec s$.

\begin{lemma}
\label{lem:equivsameblock}
Let $\vec{a},\vec{b}$ be integer vectors such that they are equivalent under $M$. Let $\vec{f}$ be a vector such that every entry $f_i$ satisfies $0 \leq f_i < 1$. Then $\vec{a}+\vec{f} \in R^{+}$ if and only if $\vec{b}+\vec{f} \in R^{+}$.
\end{lemma}
\begin{proof}
    Recall that we have a nonnegative rational vector $\vec{v} \in R^{+}$ if and only if $LM^{-1}\vec{v}$ is a integer vector. If $\vec{a} \equiv_M \vec{b}$ then we may write $\vec{b} = \vec{a} + M\vec{z}$ where $z$ is an integer vector. This gives us $LM^{-1}\vec{b} = LM^{-1}\vec{a} + L\vec{z}$ and since $L$ and $\vec{z}$ are integral, we get the desired claim.
\end{proof}

%**Example here demonstrating that if a,b are equivalent they their block structure is the same

\begin{example}
    Consider the signed graph from \cref{ex:usualbad}. We start by focusing on the chip-firing of the underlying graph (ignoring $L$ for now). From the configuration $\vec a = (1, 2, 0)$, we can fire vertex 2 to obtain the configuration $\vec b = (2, 0, 1)$, so these configurations are firing-equivalent (equivalent under $M$). Take some non-negative rational vectors where the entries are bounded above by $1$, say  $\vec f_1 = (\nicefrac 1 3, \nicefrac 2 3, 0)$ and $\vec f_2 = (\nicefrac 2 3, \nicefrac 1 3, 0)$. Looking at the images under $LM^{-1}$ we get the following table.

    
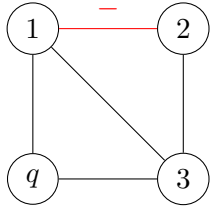
\begin{center}
        \begin{tikzpicture}[scale=2]
            % k4\e
\begin{scope}[every node/.style={circle,draw}]
\node (1) at (0,1) {$1$};
\node (2) at (1,1) {$2$};
\node (3) at (1,0) {$3$};
\node (q) at (0,0) {$q$};
\end{scope}
% positive edges
\begin{scope}[every node/.style={draw=black},
every edge/.style={draw=black}]
\path (1) edge (3);
\path (1) edge (q);
\path (3) edge (2);
\path (3) edge (q);
\end{scope}
% negative edges
\begin{scope}[every node/.style={text=red,font=\bfseries},
every edge/.style={draw=red}]
\path (1) edge node[above] {$-$} (2);
\end{scope}

        \end{tikzpicture}
    \end{center}
    
    \begin{center}
    \begin{tabular}{c | c}
        Preimage & Image under $LM^{-1}$ \\
        \hline
        $\vec a + \vec f_1 = (\nicefrac 4 3, \nicefrac 8 3, 0)$ & $(8, 7, 0)$ \\
        $\vec b + \vec f_1 = (\nicefrac 7 3, \nicefrac 2 3, 1)$ & $(7, 5, 1)$ \\
        \hline
        $\vec a + \vec f_2 = (\nicefrac 5 3, \nicefrac 7 3, 0)$ & $(8, \nicefrac{27}{4}, 0)$\\
        $\vec b + \vec f_2 = (\nicefrac 8 3, \nicefrac 1 3, 1)$ & $(7, \nicefrac{19}{4}, 1)$
    \end{tabular}
    \end{center}

    Notice from the table that we have $\vec{a} + \vec f_1$ and $\vec{b} + \vec f_1$ are both in $R^{+}$ whereas $\vec{a}+ \vec f_2$ and $\vec{b} + \vec f_2$ are both not in $R^{+}$, which is consistent with \cref{lem:equivsameblock}.
\end{example}

We now present our main result: the bijective map between superstable configurations and critical configurations of $(L,M)$-pairs that extends the classical duality map $\vec{c} \rightarrow \cm - \vec{c}$. Recall from \cref{rem:twobijects} that for $M$-matrices we have a bijection coming from equivalence classes and a bijection coming from the duality map. We still have the bijection induced by the equivalence classes \cref{thm:class}, but the classical duality map $\vec{c} \rightarrow \cm - \vec{c}$ does not work as seen in \cref{subsec:classicdualnotwork}.

Our idea is to develop a map that somewhat interpolates between the two, using the involution map $\ivo$ as a guide. In particular, when $L=M$ (the usual chip-firing on $M$-matrices), our map will be exactly the same as the classical duality map. Recall that $\sst(L,M)$ and $\crit(L,M)$ defined in \cref{subsec:lmpairstools} each represent the set of superstable preimages and critical preimages.

\begin{theorem}
\label{thm:mainbij}
    Let $(L,M)$ be any chip-firing pair. The map $\dual : \sst(L,M) \rightarrow \crit(L,M)$ given by $\ss \mapsto \cm-\ivo(\floor{\ss}) + \frp{\ss}$ is a bijection.
\end{theorem}

\begin{proof}
     We approach this proof in two cases. When $\ivo(\floor{\vec s})=\sstab(\cm-\floor{\vec s})$, we have
     % this is just the equivalence classes, hide it maybe?
     \begin{align*}
         \cm - \ivo(\floor{\vec s}) &= \cm - \sstab(\cm - \floor{\vec s}).\\
         \intertext{Since $\cm - \vec s$ is critical for any superstable configuration $\vec s$ of $M$,}
         &= \crit(\cm - (\cm - \floor{\vec s}))\\
         &= \crit(\floor{\vec s}).
     \end{align*}
     %Recall that, for $\vec a, \vec b \in \sup(G_+)$, $\vec a \Meq \vec b \implies (\vec a + \vec f \in R^+ \iff \vec b + \vec f \in R^+)$ for $\vec 0 \le \vec f < \vec 1$.
     
     From \cref{lem:equivsameblock}, since $\vec s = \floor{\vec s} + \{\vec s\} \in R^+$, we have that $\crit(\floor{\vec s}) + \{\vec s\}$ is in $R^+$, and is a critical preimage of $(L,M)$ thanks to \cref{thm:22floor}.
     
     In the second case when $\ivo(\floor{\vec s})=\floor{\vec s}$,
     recall that
     \begin{align*}
     \ivo(\floor{\vec s})=\floor{\vec s} 
        &\iff \{\LM(2\floor{\vec s})\} = \{\LM(\cm)\} 
        \\&\iff \{\LM(\floor{\vec s})\} = \{\LM(\cm - \floor{\vec s})\}.
     \end{align*}
    Together with the fact that $LM^{-1}(\floor{\vec s} + \{\vec s\})$ is integral, the above tells us that $LM^{-1}(\cm-\floor{\vec s} + \{\vec s\})$ is also integral. Hence $\cm-\floor{\vec s} + \{\vec s\} \in R^{+}$ and is a critical preimage of $(L,M)$ thanks to \cref{thm:22floor}.

Now we show that the map is injective. Assume that we have $\cm-\ivo(\floor{\ss}) + \frp{\ss} = \cm-\ivo(\floor{\ss'}) + \frp{\ss'}$. Since $\ivo(\floor{\ss}), \ivo(\floor{\ss'})$ are integral and each entry in $\frp{\ss},\frp{\ss'}$ is from $[0,1)$, we get $\ivo(\floor{\ss}) = \ivo(\floor{\ss'})$ and $\frp{\ss} = \frp{\ss'}$. Since $\ivo$ is an involution from \cref{prop:ivoinvol}, we get $\ss = \ss'$.

Finally, from \cref{thm:class}, we see that the map is a bijection.

     %Now, recall the definition of stability for some %superstable $\vec s \in R^+$.
     %begin{align*}
     %    \LM(\vec s) &\in \ZZ^n \\
     %    \{\LM(\vec s)\} &= \vec 0 \\
     %    \brace{\LM(\floor{\vec s})} + \brace{\LM(\{\vec s\})} &= \vec 0\\
     %    \brace{\LM(\cm - \floor{\vec s})} + \brace{\LM(\{\vec s\})} &= \vec 0\\
     %    \{\LM(\cm-\floor{\vec s} + \{\vec s\})\} &= \vec 0 \\
     %    \LM(\cm-\floor{\vec s} + \{\vec s\}) &\in \ZZ \\
     %\end{align*}
    %So, $\cm-\floor{\vec s} + \{\vec s\} \in R^+$. Note that $\cm-\floor{\vec s}$ is a critical configuration, so $\cm-\floor{\vec s} + \{\vec s\}$ is critical in $G_\phi$.
\end{proof}

In Question~5.2 of \cite{cho2024}, it was asked if we can establish a duality in the signed graph case. As described in \cref{ex:signed_kyle}, using $L$ as the Laplacian of a signed graph and $M$ as the Laplacian of the underlying unsigned graph, \cref{thm:mainbij} answers this question. 

\begin{example}
        Again take a look at the signed graph from \cref{ex:usualbad}.
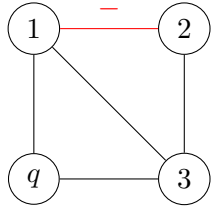
        \begin{center}
           \begin{tikzpicture}[scale=2]
               % k4\e
\begin{scope}[every node/.style={circle,draw}]
\node (1) at (0,1) {$1$};
\node (2) at (1,1) {$2$};
\node (3) at (1,0) {$3$};
\node (q) at (0,0) {$q$};
\end{scope}
% positive edges
\begin{scope}[every node/.style={draw=black},
every edge/.style={draw=black}]
\path (1) edge (3);
\path (1) edge (q);
\path (3) edge (2);
\path (3) edge (q);
\end{scope}
% negative edges
\begin{scope}[every node/.style={text=red,font=\bfseries},
every edge/.style={draw=red}]
\path (1) edge node[above] {$-$} (2);
\end{scope}

           \end{tikzpicture}
        \end{center}
        Start from the superstable configuration $(5,4,0)$ in $S^{+}$. As can be seen in the table in \cref{ex:usualbad}, its corresponding preimage is $\vec{s} = (\nicefrac{4}{3}, \nicefrac{7}{6},0)$.
        
        %We have that $\vec{s} = (\nicefrac{4}{3}, \nicefrac{7}{6},0)$ is a superstable configuration of $(L,M)$ in $R^+$. 
        We can check that $\{LM^{-1}2\floor{\vec s}\} = (0, \nicefrac{1}{2}, 0) \neq (0,0,0) = \{LM^{-1}\cmax\}$, so $\ivo (\floor{\vec s}) = \sstab(\cmax - \floor{\vec s})$. For this graph, we have $\cmax = (2,1,2)$, so $\ivo (\floor{\vec s}) = (0,0,1)$ as $M^{-1}(\cmax - \floor{\vec s} - (0,0,1)) \in \ZZ$.

        Then \[\dual(\vec s) = \cmax - \ivo(\floor{\vec s}) + \{\vec s\} = (2,1,2) - (0,0,1) + \paren{\nicefrac{1}{3}, \nicefrac{1}{6}, 0} = \paren{\nicefrac{7}{3}, \nicefrac{7}{6}, 1}\]
        gives a critical preimage in $R^+$, which corresponds to the critical configuration $(8,6,1)$ of $S^+$. The table of superstables and criticals in \cref{ex:usualbad} is aligned in a way so that the superstable configuration and the critical configuration obtained from this duality map are in the same row.
    \end{example}

When we are looking at chip-firing pairs $(M,M)$ the above duality map is exactly same as the previously known duality map for $M$-matrices: sending $\vec{c}$ to $c_{max} - \vec{c}$.

\begin{example}
    Consider the graph from \cref{ex:unsignedkyle}.
    \begin{center}
       \begin{tikzpicture}[scale=2]
           \begin{scope}[every node/.style={circle,draw}]
\node (1) at (0,1) {$1$};
\node (2) at (1,1) {$2$};
\node (3) at (1,0) {$3$};
\node (q) at (0,0) {$q$};
\end{scope}
% positive edges
    \begin{scope}[every node/.style={draw=black},
every edge/.style={draw=black}]
\path (1) edge (3);
\path (1) edge (q);
\path (3) edge (2);
\path (3) edge (q);
\path (1) edge (2);
\end{scope}
% negative edges

       \end{tikzpicture}
    \end{center}
    We have that $(0,0,1)$ and $(1,1,0)$ are superstable configurations (and also superstable preimages, since $LM^{-1}$ is the identity matrix). Since they are integral, they are fixed points in our involution $\ivo$.

    Then
    \[\dual(0,0,1) = \cm-\ivo(\floor{(0,0,1)}) + \frp{(0,0,1)} = (2,1,2) - (0,0,1) + (0,0,0) = (2,1,1),\]
    \[\dual(1,1,0) = \cm-\ivo(\floor{(1,1,0)}) + \frp{(1,1,0)} = (2,1,2) - (1,1,0) + (0,0,0) = (1,0,2),\]
    We can see that this aligns with the classical duality map between superstable and critical configurations for graphs.
\end{example}

The inverse of the above duality map looks like the following.

\begin{corollary}
    Let $(L,M)$ be any chip-firing pair. The map $\dual^{-1} : \crit(L,M) \rightarrow \sst(L,M)$ given by $\vec{c} \mapsto \ivo(\cm - \floor{\vec{c}}) + \frp{\vec{c}}$ is a bijection.
\end{corollary}

    %\begin{center}
    %    \begin{tikzpicture}[scale=2]
    %        \input{tikz/kyle-unsigned}
    %    \end{tikzpicture}
    %\end{center}
    % just start w superstables in R+ (pick 2), show where the two superstables go using this map, then show the whole list of superstable and criticals and say that it can be verified that the map works for all of them (lol)
    
    %In any all-positive graph, every superstable in $R^+$ is integral, so every superstable is a fixed point under $\ivo$, so this recovers the usual duality.

\section{Frackets}

%Recall that given an $(L,M)$-pair, the corresponding critical group is given by looking at the equivalence relation given by the matrix $L$ and each class is represented by a superstable configuration in $S^{+}$ (all are integer vectors). We can also get the same group by instead looking at their preimages in $R^+$ with the equivalence relation given by the matrix $M$ (since we are doing a linear transformation $ML^{-1}$ to go from $S^{+}$ to $R^{+}$). Hence throughout the paper, we are going to think of the critical group of $(L,M)$ with the superstable pre-images as its elements.

%\begin{example}
%Again take a look at the signed graph from \cref{ex:usualbad}. According to the definition in \cite{GuzKlivans}, the critical group here comes from the superstable configurations in $S^{+}$, using the vectors in the $1$-st column of the table and using the $+$ operation. We also get the same group, by using the vectors in the $2$-nd column of the table
%From \cref{table:toy_stats}, 
%\end{example}

    In this section, we focus on \newword{frackets}, subgroups of the critical group constructed by looking at the fractional parts of preimages. We provide a formula for calculating the cardinality of such groups and then use that result to enumerate the fixed points of the involution map studied in the previous section. 

    \subsection{The definition of frackets}
    For any integral, invertible matrix $L$, we let $\critgroup(L)$ to denote the group we get by looking at the equivalence classes of configurations given by $\equiv_L$. Let $L$ and $M$ be any integral, invertible $n$-by-$n$ matrices. We call such pair $(L,M)$ as an \newword{ii-pair}. 
    
    Furthermore, when a vector $\vec{f}$ satisfies $0 \leq f_i < 1$ for every coordinate $i$, we call it \newword{fractional vector}. Recall that for any vector $\vec{v}$, we use $\{\vec v\}$ to denote its fractional part. Combining these concepts, we can obtain the following result.
    
    \begin{definition}
        Given any ii-pair $(L,M)$ and any fractional vector $\vec f$, we define the $L$-fracket $F^L_{\vec f}$ as the subset of $\mathcal K(L)$ consisting of every equivalence class that has a vector representation $\vv \in \ZZ^n$ such that $\{ML^{-1}\vv\} = \vec{f}$.
    \end{definition}

    In other words, an $L$-fracket consists of all equivalence classes of configurations whose images under $ML^{-1}$ have the same fractional part, and the collection of $L$-frackets represent all unique fractional parts of configurations under $\ML$. It is clear from the definition that for any chip-firing pair $(L,M)$ it is also an ii-pair. Moreover, since we are only considering integral $M$-matrices, $(M,L)$ is also an ii-pair as well.
    
    \begin{example}
    Let us again consider the chip-firing pair $(L,M)$ studied in \cref{table:signed_toy_ss}.

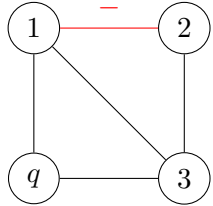
    \begin{center}
        \begin{tikzpicture}[scale=2]
        % k4\e
\begin{scope}[every node/.style={circle,draw}]
\node (1) at (0,1) {$1$};
\node (2) at (1,1) {$2$};
\node (3) at (1,0) {$3$};
\node (q) at (0,0) {$q$};
\end{scope}
% positive edges
\begin{scope}[every node/.style={draw=black},
every edge/.style={draw=black}]
\path (1) edge (3);
\path (1) edge (q);
\path (3) edge (2);
\path (3) edge (q);
\end{scope}
% negative edges
\begin{scope}[every node/.style={text=red,font=\bfseries},
every edge/.style={draw=red}]
\path (1) edge node[above] {$-$} (2);
\end{scope}

        \end{tikzpicture}
    
    \begin{tabular}{c | c}
        Element of $\critgroup(L)$ & Preimage in $R^{+}$ \\
        \hline
        $(7, 6, 8)$ & $(\nicefrac 1 2, 0, \nicefrac 5 2)$ \\
        $(8, 6, 7)$ & $(\nicefrac 5 2, 0, \nicefrac 1 2)$
    \end{tabular}
    \end{center}

    Their preimages both have fractional part $(\nicefrac 1 2, 0, \nicefrac 1 2)$, so we have $[(7,6,8)]_L, [(8,6,7)]_L \in  F^L_{(\nicefrac 1 2, 0, \nicefrac 1 2)}$.
    
    %the equivalence classes corresponding to $(7, 6, 8)$ and $(8, 6, 7)$ belong to $F^L_{(\nicefrac 1 2, 0, \nicefrac 1 2)}$.

    \begin{center}
    \begin{tabular}{c | c}
        Element of $\critgroup(M)$ & Image under $LM^{-1}$ \\
        \hline
        $(0, 0, 0)$ & $(0, 0, 0)$ \\
        $(0, 1, 0)$ & $(3, 3, 3)$
    \end{tabular}
    \end{center}
    
    The image under $LM^{-1}$ of these configurations both have fractional part $\vec 0$, so we have $[(0,0,0)]_M,[(0,1,0)]_M \in F^M_0$. 
     %We uniquely define the fracket with fractional part $\vec 0$ as the \textit{zero fracket}.
    %the equivalence classes corresponding to $(0, 0, 0)$ and $(0, 1, 0)$ belong to $F^M_0$.
    \end{example}

    %In other words, a fracket consists of all equivalence classes of signed configurations whose preimages have the same fractional part. S-frackets are complementary---an S-fracket consists of all equivalence classes of \textit{unsigned} configurations whose images have the same fractional part.

    %This last example points to a connection between $L$-frackets and $M$-frackets---a configuration $\vv \in \critgroup(M)$ is in $F^M_0$ if and only if its image $LM^{-1}\vv$ is in $F^L_0$. 

    We will call $F^L_0$ the \newword{zero fracket} from now on. Let us now establish several basic properties on frackets. In general, $F^L_{\vec{f}}$ is not necessarily a group, but the zero fracket is closed under addition:
    
    \begin{lemma}
    \label{lem:zerofrackgroup}
        The zero fracket $F^L_0$ is a subgroup of $\mathcal{K} (L)$. %The zero S-fracket $F^M_0$ is a subgroup of $\mathcal{K}(M)$.
    \end{lemma}
    %\begin{proof}
    %The zero fracket $F^L_0$ is a subset of $\mathcal{K}(L)$ closed under addition.
    %\end{proof}

    %, but we have seen in the above lemma that when $\vec{f} = \vec{0}$ we do indeed have a group.

    %Since for differing $\vec{f}$, the cosets $F^L_{\vec{f}}$ have the same size as long as they are non-empty, we get the following result.

    Since all cosets have the same size, we have the following result.

    \begin{lemma}
    \label{lem:fracketsizesame}
    For any fractional vectors $\vec{f}$ and $\vec{g}$ such that $F^L_{\vec{f}}$ and $F^L_{\vec{g}}$ are both nonempty, we have $|F^L_{\vec{f}}| = |F^L_{\vec{g}}|$.
    \end{lemma}
    %\begin{proof}
    %Suppose that $|F^L_{\vec{f}}| > 0$. For each $\vv \in F^L_{\vec{f}}$ we get $\vv+z \in F^L_{\vec{f}}$. Therefore, $|F^L_{\vec f}| \geq |F^L_0|$. Similarly, for each $\vec u \in F^L_{\vec f}$ we have $\vv - \vec u \in |F^L_0|$. This gives us $|F^L_0| \geq |F^L_{\vec f}|$. Combining the two results we get that $|F^L_{\vec{f}}| = |F^L_0|$ as long as $F^L_{\vec{f}}$ is non-empty, which implies the claim of the lemma.
    %On the other hand, we can choose an element $\vec u$ from each equivalence class in $F_{\vec f}$, and we know that $\vv - \vec u$ has fractional part $\vec 0$ and thus is a unique equivalence class in $F_0$. Therefore, $|F_0| \geq |F_{\vec f}|$.
    %\end{proof}

    We can go between the elements of $F^L_0$ and $F^M_0$ by the invertible linear transformations $ML^{-1}$ and $LM^{-1}$. This map preserves vector addition and integrality, so we get the following:
    
    \begin{proposition}\label{prop:frack_equals_s_frack}
        For an arbitrary ii-pair $(L,M)$, we have $F^L_0 \cong F^M_0$.
    \end{proposition}
    %\begin{proof}
    %   Given a configuration $\vv \in F^L_0$, we know that $ML^{-1}\vv$ is integral from the definition of frackets. Since $ML^{-1}\vv$ is integral and $LM^{-1}(ML^{-1}\vv) = \vv$ is integral, so we have $ML^{-1}\vv \in F^M_0$. Thus, $ML^{-1}$ is a bijection between $F^M_0$ and $F^L_0$ that preserves vector addition, so it is an isomorphism.
    %\end{proof}

    Recall that the cardinality of the critical group $\critgroup(L)$ is given by the determinant of $L$, denoted by $|L|$. Since $F^L_0$ is a subgroup of $\mathcal{K}(L)$ and $F^M_0$ is a subgroup of $\mathcal{K}(M)$ and from \cref{prop:frack_equals_s_frack}, we immediately see that $|F^L_0|$ divides $\gcd(|L|, |M|)$. Moreover, from \cref{lem:fracketsizesame} and \cref{prop:frack_equals_s_frack} we get the following result.

    \begin{corollary}
    \label{cor:fracketallsamesize}
    For any ii-pair $(L,M)$, all $L$-frackets and all $M$-frackets have the same size.
   \end{corollary}

    %This proposition already yields an easy bound on the size of the zero fracket. Since $F^L_0$ is a subgroup of $\mathcal{K}(L)$ and is also isomorphic to a subgroup of $\mathcal{K}(M)$, we know that $|F^L_0| \leq \gcd(|L|, |M|)$. The remaining propositions in this subsection establish more miscellaneous properties that we will use in the rest of the section.

    \begin{lemma}
    \label{lem:equivalencefrackets}
        Suppose $\vv, \vec u$ are integral vectors with $\vv \equiv_L \vec u$. Then, $\vv$ and $\vec u$ belong to the same $L$-fracket.
    \end{lemma}

    \begin{proof}
        % easier proof: if v = u in K(L), then certainly v and u lie in the same coset
    
        From $\vv \equiv_L \vec u$, we get $\vv = \vec u + L \vec z$ for some integer vector $\vec z$. Then, $ML^{-1}\vv - ML^{-1}\vec u = M \vec z$. Since $M$ is integral, $M \vec z$ is integral, so $ML^{-1} \vv - ML^{-1} \vec u$ is integral. Therefore, we get $\{ML^{-1} \vv\} = \{ML^{-1} \vec u\}$, so $\vv$ and $\vec u$ belong to the same $L$-fracket.
    \end{proof}

    Since $F^L_0$ is a subgroup of the group $\mathcal{K}(L)$, we can consider the quotient group $\mathcal{K}(L) / F^L_0$. The elements of this group are the images of the fractional vectors indexing the $L$-frackets.

    \begin{example}
    \label{ex:Lfracketquotient}
    Consider the $ii$-pair coming from the $(L,M)$ pair studied in \cref{ex:usualbad}. We can look at the 2nd column (or the 5th column) of \cref{table:toy_stats} to extract the frackets. There is an $L$-fracket corresponding to each of the following fractional vectors: 
    \[(0, 0, 0), (\nicefrac 1 3, \nicefrac 1 6, 0), (\nicefrac 2 3, \nicefrac 1 3, 0), (0, \nicefrac 1 2, 0) (\nicefrac 1 3, \nicefrac 2 3, 0), (\nicefrac 2 3, \nicefrac 5 6, 0).\]
    Beware that the image of these fractional vectors under the map $LM^{-1}$ is not necessarily an integral vector. However, there exists a vector in each of the corresponding frackets. By taking a vector from each fracket, we get the elements of $\mathcal{K}(L) / F^L_0$:
    $$(0,0,0),(3,2,0),(4,3,2),(1,1,0),(4,3,0),(5,4,2).$$
%%\begin{center}
    %\begin{tabular}{c | c}
     %  $L$-Fracket & Example Vector \\
      %  \hline
       % $(0, 0, 0)$ & $(0, 0, 0)$ \\
        %$(\nicefrac 1 3, \nicefrac 1 6, 0)$ & $(3, 2, 0)$ \\
        %$(\nicefrac 2 3, \nicefrac 1 3, 0)$ & $(4, 3, 2)$ \\
        %$(0, \nicefrac 1 2, 0)$ & $(1, 1, 0)$\\
        %$(\nicefrac 1 3, \nicefrac 2 3, 0)$ & $(4, 3, 0)$ \\
        %$(\nicefrac 2 3, \nicefrac 5 6, 0)$ & $(5, 4, 2)$
    %\end{tabular} \end{center}
    \end{example}
    
    %This group (and the group $\mathcal{K}(M) / F^M_0$) is important when computing the size of the zero fracket in the general acyclic case (and hence obtaining the size of all frackets).

    \subsection{Computing the size of the frackets}

    Recall that from \cref{cor:fracketallsamesize}, for any ii-pair $(L,M)$, all $L$-frackets and all $M$-frackets have the same size. So we only need to be able to compute the size of $F^L_0$ in order to obtain the size of all the frackets. In this subsection, we show a formula for computing $|F^L_0|$. 

    %In this subsection, we prove a theorem that characterizes the size of the zero fracket from the signed and unsigned Laplacians $L$ and $M$. 

    %TODO:
    %%Invariant factors, definition, example.
    
    We begin with a lemma that describes the largest invariant factor of each of $\mathcal{K}(L) / F^L_0$ and $\mathcal{K}(M) / F^M_0$. Given a matrix $M$ with rational entries, we use $\flcm(M)$ to denote the least common multiple of all the denominators of the entries (written in irreducible fractions). Given a matrix $M$ with integer entries, we use $\gcd(M)$ to denote the greatest common divisor of all entries. Given matrices $L,M$ with integer entries, we use $\gcd(L,M)$ to denote the greatest common divisor of all entries of $L$ and $M$.
    
    \begin{lemma}\label{lem:frackets_largest_inv_factors}
    The largest invariant factors for $\mathcal{K}(M) / F^M_0$ and $\mathcal{K}(L) / F^L_0$ are $\flcm(\LM)$ and  $\flcm(\ML)$, respectively.
    %$\frac{1}{\fgcd(\LM)}$ and  $\frac{1}{\fgcd(\ML)}$, respectively.
\end{lemma}
\begin{proof}
     Let $k = \flcm(\LM)$. Then for any integral vector $\vec{v}$, since $k\LM$ is an integral matrix, we have $k \vec{v} \in F^M_0$. We note that $k$ is the smallest integer that has this property, since we can use the unit vectors for $\vec{v}$ as well. So $k$ is the smallest value such that $k\vec{v} = 0$ for all $\vec{v} \in \critgroup(M) / F^M_0$. Thus, $k$ must be the size of the largest invariant factor of $\critgroup(M) / F^M_0$. 

    The other argument comes from applying the claim we just proved on the ii-pair $(M,L)$.
%Pick any $m \in \ZZ^{+}$ such that $m\vec{v} \in F^M_0$ for all $\vec{v} \in \critgroup(M)$. Then, for any integral vector $\vec{v}$, we know that $\LM m\vec{v}$ is integral. Therefore, $m \LM$ must be an integral matrix, so $k$ must divide $m$.
\end{proof}

\begin{example}
Let us revisit the signed graph from our running example.

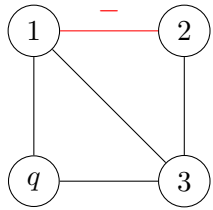
\begin{center}
        \begin{tikzpicture}[scale=2]
        % k4\e
\begin{scope}[every node/.style={circle,draw}]
\node (1) at (0,1) {$1$};
\node (2) at (1,1) {$2$};
\node (3) at (1,0) {$3$};
\node (q) at (0,0) {$q$};
\end{scope}
% positive edges
\begin{scope}[every node/.style={draw=black},
every edge/.style={draw=black}]
\path (1) edge (3);
\path (1) edge (q);
\path (3) edge (2);
\path (3) edge (q);
\end{scope}
% negative edges
\begin{scope}[every node/.style={text=red,font=\bfseries},
every edge/.style={draw=red}]
\path (1) edge node[above] {$-$} (2);
\end{scope}

        \end{tikzpicture}
\end{center}

Recall that from \cref{ex:Lfracketquotient}, the nonempty frackets are indexed by the following fractional vectors:
\[(0, 0, 0), (0, \nicefrac 1 2, 0), (\nicefrac 2 3, \nicefrac 1 3, 0), (\nicefrac 2 3, \nicefrac 5 6, 0), (\nicefrac 1 3, \nicefrac 1 6, 0), (\nicefrac 1 3, \nicefrac 2 3, 0).\]

The collection of $L$-frackets has group structure $\mathbb Z_6$. One way to see this is by considering $(\nicefrac 1 3, \nicefrac 1 6, 0)$ as a generator. Therefore, the largest invariant factor of $\critgroup(L) / \fracket^L_0$ has size 6. If we look at
\begin{gather*}
    ML^{-1} = \begin{pmatrix}
        \nicefrac 4 3 & - \nicefrac 4 3 & - \nicefrac 1 3 \\
        - \nicefrac 5 6 & \nicefrac 4 3 & - \nicefrac 1 6 \\
        0 & 0 & 1
    \end{pmatrix},
\end{gather*}

we can see that $\flcm(ML^{-1}) = 6$, which indeed matches the size of the largest invariant factor of $\critgroup(L) / \fracket^L_0$.

\end{example}

From \cref{lem:frackets_largest_inv_factors}, we can make a general statement about the size of the zero fracket.

% \begin{theorem}\label{theorem:zero_fracket_size}
%     Let $(L,M)$ be any $ii$-pair. Suppose $\mathcal{N}$ is the cardinality-wise largest group that is isomorphic to some subgroup of $\SK(M) / F^M_0$ and some subgroup of $\SK(L) / F^L_0$. Let $p$ be the product of the invariant factors of $\mathcal{N}$ excluding the largest invariant factor of $\mathcal{N}$. Then, $|\fracket^L_0| = \frac 1 p \gcd(|L|\ML,|M|\LM)$.
% \end{theorem}
\begin{theorem}\label{theorem:zero_fracket_size}
    Let $(L,M)$ be any $ii$-pair. Let $p_M$ be the product of the invariant factors of $\SK(M)/F_0^M$ excluding the largest invariant factor, and let $p_L$ be the product of the invariant factors of $\SK(M)/F_0^M$ excluding the largest invariant factor. Then, $|\fracket^L_0| = \frac{\gcd(|L|\ML,|M|\LM)}{\gcd(p_M, p_L)}$.
\end{theorem}
\begin{proof}
Let $p_M$ denote the product of the invariant factors of $\critgroup(M) / F^M_0$, excluding the largest invariant factor. Then, we see that from \cref{lem:frackets_largest_inv_factors}, together with $\flcm(LM^{-1}) =  |M| / \gcd(|M|\LM)$, we get:
    \begin{align*}
        |F^M_0| p_M &= |F^M_0| \paren{ \dfrac {| \critgroup(M) / F^M_0|} { |M| / \gcd(|M|\LM)} } \\
            &= |F^M_0| \frac{|M| / |F^M_0|}{|M| / \gcd(|M|\LM)} \\
            &= \gcd(|M|\LM).
    \end{align*}
    Similarly, we have $|\fracket^L_0|p_L = \gcd(|L|\ML)$. Using \cref{prop:frack_equals_s_frack}, we have $|\fracket^L_0|p_M = |F^M_0| p_M = \gcd(|M|\LM)$. Combining these results together, we get
    \[|\fracket^L_0| = \frac{\gcd(|L|\ML, |M|\LM)}{ \gcd(p_M, p_L)}.\]
\end{proof}
\begin{example}\label{example:naive_zero_fracket}
    Let us revisit the graph from our running example.

    
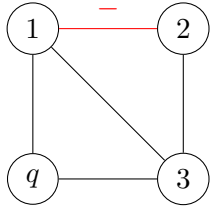
\begin{center}
        \begin{tikzpicture}[scale=2]
        % k4\e
\begin{scope}[every node/.style={circle,draw}]
\node (1) at (0,1) {$1$};
\node (2) at (1,1) {$2$};
\node (3) at (1,0) {$3$};
\node (q) at (0,0) {$q$};
\end{scope}
% positive edges
\begin{scope}[every node/.style={draw=black},
every edge/.style={draw=black}]
\path (1) edge (3);
\path (1) edge (q);
\path (3) edge (2);
\path (3) edge (q);
\end{scope}
% negative edges
\begin{scope}[every node/.style={text=red,font=\bfseries},
every edge/.style={draw=red}]
\path (1) edge node[above] {$-$} (2);
\end{scope}

        \end{tikzpicture}
    \end{center}

    \begin{multicols}{2}
    \begin{center}
    \begin{tabular}{c | c}
        $L$-Frackets & $M$-Frackets \\
        \hline
        $(0, 0, 0)$ & $(0, 0, 0)$ \\
        $(\nicefrac 1 3, \nicefrac 1 6, 0)$ & $(0, \nicefrac 1 4, 0)$ \\
        $(\nicefrac 2 3, \nicefrac 1 3, 0)$ & $(0, \nicefrac 1 2, 0)$ \\
        $(0, \nicefrac 1 2, 0)$ & $(0, \nicefrac 3 4, 0)$\\
        $(\nicefrac 1 3, \nicefrac 2 3, 0)$ & \\
        $(\nicefrac 2 3, \nicefrac 5 6, 0)$ &
    \end{tabular} \end{center}\columnbreak
    \[|L| ML^{-1} = \begin{pmatrix}
        16 & - 16 & - 4 \\
        - 10 & 16 & -2 \\
        0 & 0 & 2
    \end{pmatrix}\]\[ |M| LM^{-1} = \begin{pmatrix}
        16 & 16 & 8 \\
        10 & 16 & 6 \\
        0 & 0 & 8
    \end{pmatrix}\]
    \end{multicols}
    From the table, we see that $\mathcal{K}(M) / F^M_0 \cong \ZZ_4$, and $\mathcal{K}(L) / \fracket^L_0 \cong \ZZ_6$. Then $p_M = 1$ and $p_L = 1$ since there is only one invariant factor. We also see that the greatest common divisor of $|M|LM^{-1}$ and $|L|ML^{-1}$ is $2$, so we expect the zero fracket to have size $2$. Indeed, $F^L_0$ consists only of the equivalence classes of $(0, 0, 0)$ and $(3, 3, 3)$.
\end{example}

In the above example, we needed to compute the structure of $\mathcal{K}(M) / F^M_0$ and $\mathcal{K}(L) / F^L_0$ in order to obtain the size of the frackets. However, in many cases (conjecturally most cases for chip-firing pairs coming from graphs \cite[Conjecture 2]{clancy2015}), we can bypass this computation. If either $\mathcal{K}(M) / F^M_0$ or $\mathcal{K}(L) / \fracket^L_0$ is cyclic, we get a much simpler expression for $|\fracket^L_0|$.

%we needed a lot of information to determine the size of the frackets. The most troubling part: we needed to know the structure of $\mathcal{K}(M) / F^M_0$ and $\mathcal{K}(L) / \fracket_0$, which in some cases may be just as difficult as finding the size of the zero fracket.

%However, in many (most, probably \cite[Conjecture 2]{clancy2015}) cases, we can make a more specific statement than $\cref{theorem:zero_fracket_size}$. If either one of $\mathcal{K}(M) / F^M_0$ or $\mathcal{K}(L) / \fracket^L_0$ is cyclic, we can narrow the result to a more approachable expression for $|\fracket^L_0|$.

\begin{corollary}
\label{cor:cyclic_fracket_size}
   Choose an arbitrary ii-pair $(L, M)$. Then, $\critgroup(M) / F^M_0$ is cyclic if and only if $|\fracket^L_0| = \gcd(|M| \LM)$. Similarly, $\critgroup(L) / \fracket_0^L$ is cyclic if and only if $|\fracket^L_0| = \gcd(|L|\ML)$.
\end{corollary}
%\begin{proof}
 %   If $\critgroup(M) / F^M_0$ is cyclic, its cardinality $|M| / |F^M_0|$ is the smallest value $c$ such that $c\vec{v} = 0$ for all $\vec{v} \in \critgroup(M) / F^M_0$. It follows that $\frac{|M|}{|F^M_0|}$ is the minimal positive value $c$ such that $c\vec{v} \in F^M_0$ for all $\vec{v} \in \critgroup(M)$. Also, since $\frac{|M|}{\gcd(|M|\LM)} \LM$ is integral, we have $\frac{|M|}{\gcd(|M|\LM)}\vec{v} \in F^M_0$ for any $\vec{v} \in \critgroup(M)$. Hence, we see that $\frac{|M|}{|F^M_0|}$ divides $\frac{|M|}{\gcd(|M|\LM)}$. Therefore, $\gcd(|M|\LM)$ divides $|F^M_0| = |\fracket^L_0|$.

%Since $\frac{|M|}{|F^M_0|}$ is the minimal value $c$ such that $cg \in F^M_0$ for all $g \in \critgroup(M)$,

    %From \cref{theorem:zero_fracket_size}, we find that $|\fracket^M_0|$ also divides $\gcd(|M| \LM)$, so we can conclude that $|\fracket^L_0| = \gcd(|M|\LM)$.

%\end{proof}
\begin{proof}
    If $\SK(M)/F_0^M$ is cyclic, then $|\SK(M)/F_0^M| = \flcm(LM^{-1})$ by \cref{lem:frackets_largest_inv_factors}. Thus,

    \begin{gather*}
        \frac{|M|}{|F_0^M|} = |\SK(M)/F_0^M| = \flcm(LM^{-1}) = \frac{|M|}{\gcd(|M|LM^{-1})},
    \end{gather*}
    establishing that $|F_0^M| = \gcd(|M|LM^{-1})$.

    For the other direction of the proof, suppose $|\fracket^M_0| = \gcd(|M| \LM)$. Then, $|\critgroup(M) / F^M_0| = |M| / \gcd(|M| \LM)$. By \cref{lem:frackets_largest_inv_factors}, the largest invariant factor of $\critgroup(M) / F^M_0$ also has size $\flcm (LM^{-1}) = |M| / \gcd(|M| \LM)$. Therefore, $\critgroup(M) / F^M_0$ has exactly one invariant factor, which means it is cyclic.
\end{proof}

    %Suppose $\critgroup(G_\phi)$ is cyclic. Then, $\critgroup(G_\phi) / B_0$ is cyclic with size $\frac{|L|}{|B_0|}$. Thus, $\frac{|L|}{|B_0|}$ is the minimal positive value $c$ such that $cg \in B_0$ for all $g \in \critgroup(G_\phi)$. We know that $\ML \frac{|L|}{\gcd(|L|\ML)}$ is integral. Therefore, $\frac{|L|}{\gcd(|L|\ML)}g \in B_0$ for any $g \in \critgroup(G_\phi)$. Since $\frac{|M|}{|B_0|}$ is the minimal value $c$ such that $cg \in B_0$ for all $g \in \critgroup(G_\phi)$, we see that $\frac{|L|}{|B_0|}$ must divide $\frac{|L|}{\gcd(|L|\ML)}$. Therefore, $\gcd(|L|\ML)$ must divide $|B_0|$. We conclude by Lemma 1 that $|B_0| = \gcd(|L|\ML)$.

Therefore, if there somehow is a guarantee that either one of $\mathcal{K}(M) / F^M_0$ or $\mathcal{K}(L) / \fracket^L_0$ is cyclic, \cref{cor:cyclic_fracket_size} gives us a way to enumerate the size of the frackets in a very simple manner. For example, consider the ii-pair $(L,M)$ coming from a signed graph. If we know that the underlying graph has a cyclic critical group $\mathcal{K}(M)$, then $\mathcal{K}(M) / F^M_0$ must also be cyclic.
%, where $M$ stands for the Laplacian of the unsigned graph.

\begin{example}\label{example:soph_zero_fracket}
    Let us revisit \cref{example:naive_zero_fracket} using \cref{cor:cyclic_fracket_size}, given that the underlying graph's critical group is cyclic. All we need is to compute $|M|LM^{-1}$ (which we have already done in \cref{example:naive_zero_fracket}). We see that $\gcd(|M|LM^{-1}) = 2$, which matches $|\fracket_0^L|$.
\end{example}

%TODO feels a bit disjoint from the rest of the paper
%\begin{remark}
%Computationally it seems that \cite[Conjecture 2]{clancy2015} holds
%over random signed graphs as well (with \cite[Corollary 9.5, Remark 9.7]{wood2017} 
%proving that it is an upper bound). 
%\end{remark}

\subsection{Analyzing the number of fixed points in the involution map}

In \cref{def:ivomap} we defined an involution $\ivo$ on the set of superstable configurations of an $M$-matrix, given a chip-firing pair $(L,M)$. As discussed just before \cref{thm:mainbij}, the map $\ivo$ dictates how close our duality map $\chi$ (defined in \cref{thm:mainbij}) is to the usual bijection coming from equivalence classes or the classical duality map. In particular, when $L=M$ (the chip-firing on $M$-matrices) the map $\ivo$ is simply the identity map (so every element of $\ss(M)$ is a fixed point) and our map $\chi$ becomes the classical duality map. Hence, the number of fixed points measures how close we are to the two bijections: if it is high, our map $\chi$ is close to the classical duality map, whereas if it is low, our map is close to the bijection coming from equivalence classes.

In this section, we use the techniques developed in the previous subsections to count the number of fixed points of $\ivo$. 
   %\begin{definition}
    %Given an $(L,M)$ pair, a fixed point is an unsigned superstable configuration $\vec s$ satisfying 
    %\[\LM(\cm - 2\ss) \in \ZZ^{G\setminus q}.\]
%\end{definition}
Recall that the fixed points of $\ivo$ are $\vec s \in \sstab(M)$ such that $\{\LM(\cm - 2\ss)\} = \vec{0}.$ 

\begin{theorem}\label{thm:numsols}
    The number of $\ss \in \sstab(M)$, satisfying $\cm - 2\ss \in F^M_0$ 
    is equal to either $0$ or $|F^M_0|d$, where $d$ is the number of elements of $\mathcal{K}(M) / F^M_0$ with order at most 2.
\end{theorem}
\begin{proof} %%%TODO : This proof   
    Suppose that there exists some $s \in \sst(M)$ such that $\cm - 2s \in F^M_0$. Note that every element $g \in \SK(M)$ is uniquely expressible as $s + h + f$ for some $h \in \SK(M)/F^M_0$, $f \in F^M_0$. Then, $\cm - 2g = (\cm - 2s) - 2f - 2h$.

    Since $\cm - 2s \in F^M_0$ and $2f \in F^M_0$, we have that $\cm - 2g \in F^M_0$ 
    if and only if $2h \in F^M_0$. There are $|F_0^M|$ possible choices for $f$, and $d$ possible choices for $h$, so there are $|F_0^M|d$ such options for $g$.
\end{proof}

\begin{example}
Let us revisit the running example of a $(L,M)$-pair coming from a signed graph.
    \begin{center}
        \begin{tabular}{c | c}
        $\sstab(M)$ & Its image (under $LM^{-1}$) \\
        \hline
        $(0, 0, 0)$ & $(0, 0, 0)$ \\
        $(0, 0, 1)$ & $(1, \nicefrac 3 4, 1)$ \\
        $(0, 0, 2)$ & $(2, \nicefrac 3 2, 2)$ \\
        $(0, 1, 0)$ & $(2, 2, 0)$ \\
        $(0, 1, 1)$ & $(3, \nicefrac{11}{4}, 1)$ \\
        $(1, 0, 0)$ & $(2, \nicefrac 5 4, 1)$ \\
        $(1, 1, 0)$ & $(4, \nicefrac{13}{4}, 1)$ \\
        $(2, 0, 0)$ & $(4, \nicefrac 5 2, 2)$ \\
        \hline
        $\cm = (2, 1, 2)$ & $(8, 6, 4)$
        \end{tabular}
    \end{center}

    From the above table, we see that $\cm \in F^M_0$. Therefore, the $4$ superstable preimages that $2\ss$ is in the zero fracket of $M$ will be the fixed point of $\mu$.
    %Therefore, any of the $4$ superstable preimages $\ss$ that lie in an order-1 or order-2 S-fracket will be a fixed point, since these configurations satisfy $\cm - 2\ss \in F^M_0$.

    Recall that $\mathcal{K}(M) = \ZZ_8$ and $|F^M_0| = 2$ (from \cref{example:soph_zero_fracket}). Therefore, $\mathcal{K}(M) / F^M_0 \cong \ZZ_4$, which has $2$ elements of order at most $2$. Using $\cref{thm:numsols}$ with $d = 2$ and $|F^M_0| = 2$, we can verify that we get $4$ fixed points.
\end{example}

The remainder of this subsection will be on other more specific observations and corollaries of \cref{thm:numsols} and \cref{cor:cyclic_fracket_size}.

\begin{proposition}\label{oddorder}
    If $\cm$ has odd order in $\mathcal{K}(M) / F^M_0$, then the number of fixed points of $\ivo$ is nonzero.
\end{proposition}
\begin{proof}
    Let $H$ be the subgroup of $\mathcal{K}(M) / F^M_0$ generated by $\cm$. 
    Since $\cm$ has odd order, we know that $|H|$ is odd. 
    There is no element of $H$ with order 2, so it follows that there is some element $g \in H$ such that $2g = \cm$. Then, $g$ is a fixed point.
\end{proof}

If $\mathcal{K}(M) / F^M_0$ is cyclic and $\cm$ has even order, we have an exact criterion to check whether the number of fixed points is zero or not.

\begin{corollary}
    Suppose $\mathcal{K}(M) / F^M_0$ is cyclic and $\cm$ has even order in $\mathcal{K}(M) / F^M_0$. Then, the number of fixed points of $\ivo$ is nonzero if and only if $\frac{|\mathcal{K}(M) / F^M_0|}{\ord(\cm)}$ is even.
\end{corollary}
\begin{proof}
    Suppose $g$ generates $\mathcal{K}(M) / F^M_0$, so $\cm = kg$ for some positive integer $k$. Since $\ord(\cm) k g = 0$ and $\mathcal{K}(M) / F^M_0$ is cyclic, we know that $|\mathcal{K}(M) / F^M_0|$ divides $\ord(\cm)k$. Therefore, $|\mathcal{K}(M) / F^M_0| / \ord(\cm)$ must divide $k$.

    Since $\frac{|\mathcal{K}(M) / F^M_0|}{\ord(\cm)}$ is even, we see that $k$ must be even as well. Therefore, $(k/2)g$ is a fixed point of $\ivo$.
    
    For the other direction of the proof, assume there is some fixed point $g$ of $\ivo$. then we have $2g = \cm$ in $\mathcal{K}(M) / F^M_0$. We know that $2 \ord(\cm)g = 0$, so $\ord(g)$ must divide $2 \ord(\cm)$. Let $H$ be the subgroup of $\mathcal{K}(M) / F^M_0$ generated by $g$. Then, $\cm \in H$ and $|H|$ cyclic, so $\ord(\cm)$ must divide $|H| = \ord(g)$. So we have $\ord(g) | 2\ord(\cm)$ and $\ord(\cm) | \ord(g)$. Thus, one of the two happens: $\ord(g) = \ord(\cm)$ or $\ord(g) = 2 \ord(\cm)$.

    %Next, we show that $ord(g) = 2ord(\cm)$. 
    Suppose for the sake of contradiction that $\ord(g) = \ord(\cm)$. In other words, $\ord(g) = \ord(2g)$. This means that $\ord(g)$ is odd. By our assumption that $\ord(g) = \ord(\cm)$, it follows that $\ord(\cm)$ is also odd, contradicting our assumption that $\cm$ is of even order.

    Therefore, $\ord(g) = 2\ord(\cm)$. Then $\frac{|\mathcal{K}(M) / F^M_0|}{\ord(\cm)} = 2 \frac{|\mathcal{K}(M) / F^M_0|}{\ord(g)}$, so $\frac{|\mathcal{K}(M) / F^M_0|}{\ord(\cm)}$ must be even.
\end{proof}

%When $\mathcal{K}(M) / F^M_0$ is not cyclic, we cannot use the above criterion, but in the small cases where the order of $\cm$ satisfies a very specific condition, we can say something above the existence of fixed points of $\ivo$.

%\begin{corollary}
    %If $\cm$ has even order that is equal to the largest invariant factor of 
    %$\mathcal{K}(M) / F^M_0$, then there are no fixed points.
%\end{corollary}

%\begin{proof}
    %Suppose $\cm$ is of even order that is equal to the largest invariant factor of $\mathcal{K}(M) / F^M_0$. 
    %Suppose for the sake of contradiction that there is some fixed point $g$ of $\ivo$, so we have that $2g = \cm$. Because the order of $\cm$ is even, we see that $ord(g) = 2 ord(\cm)$. Therefore, the order of $g$ must be double the size of the largest invariant factor, which contradicts the fact that the largest invariant factor gives the largest order of any element in the group.
%\end{proof}

%\begin{corollary}
    %If $\cm$ has odd order that is equal to the largest invariant factor of 
    %$\mathcal{K}(M) / F^M_0$, then the number of fixed points is $|F^M_0|$.
%\end{corollary}

%\begin{proof}
    %Using \cref{oddorder}, we know that there must be fixed points, and using \cref{thm:numsols}, we know there must be $|F^M_0| \cdot 2^{\text{\# of even invariant factors of } \mathcal{K}(M) / F^M_0}$ fixed points. Since the largest invariant factor is odd, it follows that the number of even invariant factors of $\mathcal{K}(M)$ must be 0. Therefore, the number of solutions is exactly $|\fracket_0|$.
%\end{proof}

\subsection{Chip-firing pairs coming from complete graphs}
In this subsection, we will be focusing on chip-firing pairs coming from signed graphs, where the underlying graph is the complete graph. Same as the chip-firing pairs coming from signed graphs we have seen before, $L$ is going to denote the (reduced) Laplacian of the signed graph and $M$ is going to denote the (reduced) Laplacian of the underlying complete graph $K_n$, where $n$ is even.

\begin{lemma}\label{KLMinv.den}
    For such $(L,M)$ pairs coming from a complete graph $K_n$ where $n$ is even,
    \[\frac{n}{2}\LM \text{ is an integer matrix.}\]
\end{lemma}
\begin{proof}
    We will show that every term in $\LM$ has a denominator dividing $n/2$. The entries of the inverse of $M$ can be described as the following:
    \[\Mi_{i,j} = \begin{cases}
        \frac{2}{n} & \text{ if } i = j \\
        \frac{1}{n} & \text{ if } i \neq j.
    \end{cases}\]

    From this we analyze the entry $(LM^{-1})_{i,j} = \sum_{k = 1}^{n-1} L_{i,k}\Mi_{k,j}$. Let $\alpha_i$ to denote the number of $-1$'s that appear
    in row $i$. Then there will be exactly $n-2-\alpha_i$ many $+1$'s in that row. We do a case-by-case analysis. 

    In the case $i=j$, we have that 
    $$(\LM)_{i,j} =\sum_{k = 1}^{n-1} L_{i,k}\Mi_{k,j}= \frac{2(n-1)}{n}-\frac{\alpha_i}{n}+\frac{n-2-\alpha_i}{n} = \frac{3n-4-2\alpha_i}{n}.$$ If $L_{i,j} = -1$, then we have $$(\LM)_{i,j} = \sum_{k = 1}^{n-1} L_{i,k}\Mi_{k,j}= \frac{n-1}{n}-\frac{2}{n} - \frac{\alpha_i-1}{n} +\frac{n-2-\alpha_i}{n} = \frac{2n-4-2\alpha_i}{n}.$$ Finally when $L_{i,j} = 1$ we have
    \[(\LM)_{i,j} = \sum_{k = 1}^{n-1} L_{i,k}\Mi_{k,j}= \frac{n-1}{n} + \frac{2}{n} - \frac{\alpha_i}{n} +\frac{n-3-\alpha_i}{n} = \frac{2n-2-2\alpha_i}{n}.\]    
    Since $n$ is even, the numerators in all cases are also even. So $\frac{n}{2}\LM$ is an integer matrix.
\end{proof}

Using the above, we now show that any signed graph of a complete graph $K_n$ where $n$ is a fixed even number, contain a nontrivial common subgroup. Actually we show something stronger, that the zero frackets (a subgroup of the critical group) of any signed graph with underlying graph being the complete graph $K_n$ where $n$ is even, contain a nontrivial common subgroup.

\begin{theorem}
\label{thm:z0common}
The zero fracket $F_0^L$ of any $(L,M)$-pair coming from some signed graph where the underlying graph is the complete graph $K_n$ where $n$ is even, has a subgroup isomorphic to $\ZZ_2^{n-2}$.
\end{theorem}
\begin{proof}
    Consider the configuration $\vec s_i := \frac{n}{2} \vec{e_i}$. Then by 
    \cref{KLMinv.den}, we have that $\LM\paren{\frac{n}{2}\vec{e_i}}$ is an integer vector, so this is a 
    valid configuration in $R^+$. From \cref{thm:22floor}, we have that $\vec s_i$ is a superstable perimage. Additionally, $n \vec{e_i} = M(\vec{1} + \vec{e_i})$, so $\vec s_i$ has order 2 in $F_0^L$.

    For each $I \subseteq \{1\dots n\}$ such that 
    $0 \leq |I| \leq \frac{n-2}{2}$, we know that
    $\sum_{i \in I} s_i$ is a superstable preimage: when we fire some set $S$, for any $v \in S$ we start with either zero or $\frac{n}{2}$ chips, will lose $n-1$ chips, and gain back at most $|I|-1 \leq \frac{n-4}{2}$ chips, resulting in that vertex having at most $-1$ chips after the firing.
    And since $\sum_{i \in I} s_i$ is a sum of elements of order $2$, it has order $2$ as well in $F_0^L$.

    There are $\frac{2^{n-1}}{2}$ possible choices for $I$ such that $0 \leq |I| \leq \frac{n-2}{2}$, so the set of $\sum_{i \in I} s_i$'s form a subgroup of $F_0^L$ isomorphic to $\ZZ_2^{n-2}$.
\end{proof}

% Write all critical groups of signed graphs coming from $K_6$?? and show that they all have $Z_2^4$ as a subgroup.

% Here are the critical groups:
% Z6^4
% Z2 + Z2 + Z12 + Z36
% Z2 + Z2 + Z6 + Z78
% Z2 + Z2 + Z10 + Z50
% Z2 + Z2 + Z8 + Z64
% Z2 + Z2 + Z4 + Z132
% Z4 + Z4 + Z4 + Z36

\begin{example}
    Consider the graph $K_6$. Over all possible signed graphs coming from $K_6$, there are seven possible critical groups up to isomorphism:
\begin{gather*}
    \mathbb Z_6 \oplus \ZZ_6 \oplus \ZZ_6 \oplus \ZZ_6, \\
    \ZZ_{36} \oplus \ZZ_{12} \oplus \ZZ_2 \oplus \ZZ_2, \\
    \ZZ_{78} \oplus \ZZ_6 \oplus \ZZ_2 \oplus \ZZ_2, \\
    \ZZ_{50} \oplus \ZZ_{10} \oplus \ZZ_{2} \oplus \ZZ_{2}, \\
    \ZZ_{64} \oplus \ZZ_8 \oplus \ZZ_2 \oplus \ZZ_{2}, \\
    \ZZ_{132} \oplus \ZZ_4 \oplus \ZZ_2 \oplus \ZZ_{2}, \\
    \ZZ_{36} \oplus \ZZ_4 \oplus \ZZ_4 \oplus \ZZ_4.
\end{gather*}
We can check that $\mathbb Z_2^4$ is a subgroup of all cases above.
\end{example}

% \section{Computing Superstables}
% (in polyhedra.tex)

\section{Further questions}
% NOTE: To be polished
In this section, we discuss the remaining questions that naturally follow this study.

%One question that remains revolves around the central duality presented in Section 3. This duality defines a unique mapping between the superstable and critical groups using an involution $\mu$ on $\sstab(M)$. However, certain configurations (namely, those configurations with $\mu(\floor{\vec s}) = \sstab(\cmax - \floor{\vec s})$) must still rely on the bijection defined by the equivalence classes to find a unique mapping.

The duality between $z$-superstable and critical configurations of $(L,M)$-pairs constructed in Section $3$, relied on the involution $\ivo$ on the set of superstable configurations of $M$. When we constructed $\ivo$, we used the equivalence relation given by $M$. It would be interesting if we can skip that process as well.

\begin{question}
Can one construct a duality between the set of $z$-superstable configurations and the set of critical configuration of chip-firing pairs without ever relying on the equivalence class computation (of $L$ or $M$)?    
\end{question}

A lot of examples we used came from signed graphs and their Laplacian. It would be interesting if we can generalize \cref{thm:z0common} to more general graphs:

\begin{question}
Let $(L,M)$ be a chip-firing pair coming from a signed graph with an underlying graph $G$. Is there a way to obtain the biggest group that all critical groups of the chip-firing pairs coming from $G$ contain as a subgroup up to isomorphism?
\end{question}

For signed graphs, there is an operation called \newword{vertex-switching} \cite{Zaslavsky}. It is known that there is an isomorphism between the critical groups of signed graphs that are switching equivalent to one another. So it is natural to ask the following question:

\begin{question}
 Can one construct a natural isomorphism between the critical groups of switching equivalent $(L, M)$ pairs (without any reliance on the obvious bijection coming from the equivalence classes)?
\end{question}

Recall that in \cref{thm:numsols} we could only describe the number of fixed
points of $\ivo$ when we were guaranteed that it was nonzero. It would be interesting if we have an elegant way to describe exactly when the number of fixed points of $\ivo$ is nonzero in terms of $L$ and $M$:
\begin{question}
Is there a way to describe exactly when the map $\ivo$ has zero fixed points?
\end{question}

Lastly, although in this paper we strictly focused on $M$-matrices having integer entries, the chip-firing pairs allow $M$-matrices having real entries (the matrix $L$ still has to be integral). It would be interesting if one can come up with a duality for non-integral $M$-matrices, and extend it to chip-firing pairs in a similar way:

\begin{question}
Is there a nice duality map between critical and superstable configurations for non-integral $M$-matrices? 
\end{question}

\subsection*{Acknowledgements}
 The work presented here was conducted as part of an REU at Texas State University in the summer of 2024, sponsored by NSF. We thank NSF and Texas State for the support and the stimulating work environment.
    
\printbibliography
\end{document}